\documentclass[preprint,12pt]{elsarticle}

\biboptions{numbers,sort&compress}
\usepackage{amssymb}
\usepackage{amsthm}
\usepackage{amsmath}
\usepackage{epic}
\usepackage{CJK}
\usepackage{setspace}
\usepackage{bm}
\newtheorem{thm}{Theorem}[section]
\newtheorem{cor}[thm]{Corollary}
\newtheorem{lem}[thm]{Lemma}
\newtheorem{pro}[thm]{Proposition}

\newtheorem{rem}{Remark}

\journal{~}

\begin{document}
\begin{spacing}{1.15}
\begin{CJK*}{GBK}{song}
\begin{frontmatter}
\title{\textbf{Spherical two-distance sets and graph eigenvalues}}

\author{Jiang Zhou}\ead{zhoujiang@hrbeu.edu.cn}

\address{College of Mathematical Sciences, Harbin Engineering University, Harbin 150001, PR China}

\begin{abstract}
A set of unit vectors in $\mathbb{R}^d$ is a called a spherical two-distance set if the inner products of distinct vectors only take two values. In this paper, we give explicit correspondence between spherical two-distance sets and graphs with specific spectral properties, and derive some bounds on the maximum size of spherical two-distance sets.
\end{abstract}

\end{frontmatter}

\section{Introduction}
A set of unit vectors in $\mathbb{R}^d$ is a called a \textit{spherical two-distance set} if the inner products of distinct vectors only take two values \cite{Glazyrin,Jiang2,Musin}. The problem of determining the maximum size of spherical two-distance sets is an important topic in discrete geometry. Delsarte, Goethals and Seidel \cite{Delsarte} proved that the maximum size of spherical two-distance sets in $\mathbb{R}^d$ does not exceed $\frac{1}{2}d(d+3)$. In \cite{Glazyrin}, Glazyrin and Yu showed that the maximum size of spherical two-distance sets in $\mathbb{R}^d$ is $\frac{1}{2}d(d+1)$ when $d+3\geq10$ is not an odd perfect square.

For $-1\leq\beta<\alpha<1$, a \textit{spherical $\{\alpha,\beta\}$-code} $S$ is a spherical two-distance set such that $x^\top y\in\{\alpha,\beta\}$ for any two distinct $x,y\in S$. Let $N_{\alpha,\beta}(d)$ denote the maximum size of a spherical $\{\alpha,\beta\}$-code in $\mathbb{R}^d$. A spherical $\{\alpha,-\alpha\}$-code is a set of \textit{equiangular lines} with pairwise angle $\arccos\alpha$. From some recent work \cite{Balla,Bukh,Jiang2}, we know that there exists constant $c_{\alpha,\beta}$ such that $N_{\alpha,\beta}(d)\leq c_{\alpha,\beta}d$ for fixed $-1\leq\beta<0\leq\alpha<1$ and sufficiently large $d$. By using the spectral method, Jiang et al. \cite{Jiang1} determined $N_{\alpha,-\alpha}(d)$ for fixed $\alpha$ and sufficiently large $d$.

We will investigate spherical $\{\alpha,\beta\}$-code without requiring fixed $\alpha$ and $\beta$. The rest of this paper is organized as follows. In Section 2, we introduce some notations and auxiliary lemmas. In Sections 3 and 4, we give explicit correspondence between spherical two-distance sets and graphs. In Section 5, some bounds on $N_{\alpha,\beta}(d)$ are obtained by considering the structural and spectral properties of graphs associated with spherical two-distance sets.

\section{Preliminaries}
Let $V(G)$ and $E(G)$ denote the vertex set and the edge set of a graph $G$, respectively. For an $n$-vertex graph $G$, the \textit{adjacency matrix} $A_G$ of $G$ is an $n\times n$ symmetric matrix with entries
\begin{eqnarray*}
(A_G)_{ij}=\begin{cases}1~~~~~~\mbox{if}~\{i,j\}\in E(G),\\
0~~~~~~\mbox{if}~\{i,j\}\notin E(G).\end{cases}
\end{eqnarray*}
Eigenvalues of $A_G$ are called \textit{eigenvalues} of $G$. Let $\lambda_1(G)\geq\cdots\geq\lambda_n(G)$ denote the eigenvalues of $G$, that is, $\lambda_j(G)$ is the $j$-th largest eigenvalue of $G$. For an eigenvalue $\lambda$ of $G$ with multiplicity $m$, the \textit{codimension} of $\lambda$ in $G$ is defined as $n-m={\rm rank}(\lambda I-A_G)$.
\begin{lem}\label{lem2.1}
\textup{\cite{Cvetkovic1}} Let $G$ be an $n$-vertex connected graph. Then
\begin{eqnarray*}
\lambda_n(G)\geq-\sqrt{\left\lfloor\frac{n}{2}\right\rfloor\left\lceil\frac{n}{2}\right\rceil}.
\end{eqnarray*}
\end{lem}

Let $R(M)=\{x:x=My,y\in\mathbb{R}^n\}$ denote the range (column space) of an $n\times n$ real matrix $M$. For a real symmetric matrix, let ${\rm in}_+(M)$ and ${\rm in}_-(M)$ denote the number of positive and negative eigenvalues of $M$, respectively.
\begin{lem}\label{lem2.4}
\textup{\cite[Lemma 4]{Gregory}} Let $M$ be an $n\times n$ real symmetric matrix, and let $u\in\mathbb{R}^n$. For $c<0$, the following hold:\\
(1) If $u\notin R(M)$, then
\begin{eqnarray*}
{\rm in}_+(M+cuu^\top)={\rm in}_+(M),~{\rm in}_-(M+cuu^\top)={\rm in}_-(M)+1.
\end{eqnarray*}
(2) If $u=Mx$ for some $x\in\mathbb{R}^n$ and $cx^\top u>-1$, then
\begin{eqnarray*}
{\rm in}_+(M+cuu^\top)={\rm in}_+(M),~{\rm in}_-(M+cuu^\top)={\rm in}_-(M).
\end{eqnarray*}
(3) If $u=Mx$ for some $x\in\mathbb{R}^n$ and $cx^\top u<-1$, then
\begin{eqnarray*}
{\rm in}_+(M+cuu^\top)={\rm in}_+(M)-1,~{\rm in}_-(M+cuu^\top)={\rm in}_-(M)+1.
\end{eqnarray*}
(4) If $u=Mx$ for some $x\in\mathbb{R}^n$ and $cx^\top u=-1$, then
\begin{eqnarray*}
{\rm in}_+(M+cuu^\top)={\rm in}_+(M)-1,~{\rm in}_-(M+cuu^\top)={\rm in}_-(M).
\end{eqnarray*}
\end{lem}

\begin{lem}\label{lem2.5}
\textup{\cite[Lemma 4]{Gregory}} Let $M$ be an $n\times n$ real symmetric matrix, and let $u\in\mathbb{R}^n$. For $c>0$, the following hold:\\
(1) If $u\notin R(M)$, then
\begin{eqnarray*}
{\rm in}_+(M+cuu^\top)={\rm in}_+(M)+1,~{\rm in}_-(M+cuu^\top)={\rm in}_-(M).
\end{eqnarray*}
(2) If $u=Mx$ for some $x\in\mathbb{R}^n$ and $cx^\top u>-1$, then
\begin{eqnarray*}
{\rm in}_+(M+cuu^\top)={\rm in}_+(M),~{\rm in}_-(M+cuu^\top)={\rm in}_-(M).
\end{eqnarray*}
(3) If $u=Mx$ for some $x\in\mathbb{R}^n$ and $cx^\top u<-1$, then
\begin{eqnarray*}
{\rm in}_+(M+cuu^\top)={\rm in}_+(M)+1,~{\rm in}_-(M+cuu^\top)={\rm in}_-(M)-1.
\end{eqnarray*}
(4) If $u=Mx$ for some $x\in\mathbb{R}^n$ and $cx^\top u=-1$, then
\begin{eqnarray*}
{\rm in}_+(M+cuu^\top)={\rm in}_+(M),~{\rm in}_-(M+cuu^\top)={\rm in}_-(M)-1.
\end{eqnarray*}
\end{lem}



\section{The associated $\alpha$-graphs of spherical two-distance sets}
The \textit{$\{1\}$-inverse} of a matrix $A$, denoted by $A^{(1)}$, is a matrix $X$ such that $AXA=A$. For a square matrix $M$, the \textit{group inverse} of $M$, denoted by $M^\#$, is the matrix $X$ such that $MXM=M,~XMX=X$ and $MX=XM$. It is known \cite{Ben-Israel} that $M^\#$ exists if and only if $\mbox{\rm rank}(M)=\mbox{\rm rank}(M^2)$. If $M^\#$ exists, then $M^\#$ is unique.

Let ${\rm j}$ and $J$ denote the all-ones column vector and the all-ones matrix, respectively. Let $S=\{u_1,\ldots,u_n\}$ be a spherical $\{\alpha,\beta\}$-code ($-1\leq\beta<\alpha<1$). The \textit{rank} of $S$ is defined as the dimension of the space spanned by $S$, which equals to the rank of the Gram matrix of $S$. Let $\Gamma_\alpha(S)$ denote the \textit{associated $\alpha$-graph} with vertex set $S$ and edge set $E(\Gamma_\alpha(S))=\{u_iu_j:u_i^\top u_j=\alpha\}$. We give the correspondence between spherical two-distance sets and $\alpha$-graphs as follows.
\begin{thm}\label{thm3.1}
Suppose that $-1\leq\beta<\alpha<1$, $\beta<0$, and take $\mu=\frac{1-\beta}{\alpha-\beta}$. If $S$ is a spherical $\{\alpha,\beta\}$-code with size $n$ and rank $r$, then the following hold:\\
(1) $\lambda_n(\Gamma_\alpha(S))\geq-\mu$.\\
(2) Let $A$ be the adjacency matrix of $\Gamma_\alpha(S)$, then
\begin{eqnarray*}
{\rm j}\in R(A+\mu I)&,&{\rm j}^\top(A+\mu I)^\#{\rm j}\leq\frac{\alpha-\beta}{-\beta},\\
{\rm rank}(A+\mu I)&=&\begin{cases}r+1~~~~\mbox{if}~~{\rm j}^\top(A+\mu I)^\#{\rm j}=\frac{\alpha-\beta}{-\beta},\\
r~~~~~~~~~\mbox{otherwise}.\end{cases}
\end{eqnarray*}
Conversely, if an $n$-vertex graph $G$ satisfies (1) and (2), then there exists a spherical $\{\alpha,\beta\}$-code $S$ with rank $r$ such that $\Gamma_\alpha(S)=G$.
\end{thm}
\begin{proof}
Suppose that $S=\{u_1,\ldots,u_n\}$ is a spherical $\{\alpha,\beta\}$-code with rank $r$, take $\mu=\frac{1-\beta}{\alpha-\beta}$. Let $U=(u_1,\ldots,u_n)$ be the matrix whose $i$-th column is $u_i$, then the Gram matrix of $S$ is
\begin{eqnarray*}
U^\top U=(\alpha-\beta)(A+\mu I)+\beta J=(\alpha-\beta)\left(A+\mu I+\frac{\beta}{\alpha-\beta}{\rm j}{\rm j}^\top\right).
\end{eqnarray*}
By $\beta<0$ and $\beta<\alpha$, we know that $A+\mu I=(\alpha-\beta)^{-1}(U^\top U-\beta J)$ is positive semidefinite. Hence $\lambda_n(\Gamma_\alpha(S))\geq-\mu$ and
\begin{eqnarray*}
{\rm in}_-(A+\mu I)={\rm in}_-(U^\top U)=0,~{\rm rank}(A+\mu I)={\rm in}_+(A+\mu I),~{\rm in}_+(U^\top U)=r.
\end{eqnarray*}
Since ${\rm in}_-(A+\mu I)={\rm in}_-(U^\top U)$, by $U^\top U=(\alpha-\beta)(A+\mu I+\frac{\beta}{\alpha-\beta}{\rm j}{\rm j}^\top)$ and Lemma \ref{lem2.4}, we know that part (2) or (4) in Lemma \ref{lem2.4} holds, and in both cases we have ${\rm j}\in R(A+\mu I)$. Then ${\rm j}=(A+\mu I)x$ for some $x\in\mathbb{R}^n$ and $(A+\mu I)(A+\mu I)^\#{\rm j}={\rm j}$, because $(A+\mu I)(A+\mu I)^\#(A+\mu I)=A+\mu I$ from the definition of the group inverse. So we have
\begin{eqnarray*}
x^\top{\rm j}=x^\top(A+\mu I)(A+\mu I)^\#{\rm j}={\rm j}^\top(A+\mu I)^\#{\rm j}.
\end{eqnarray*}
If part (2) in Lemma \ref{lem2.4} holds, then
\begin{eqnarray*}
r={\rm in}_+(U^\top U)={\rm in}_+(A+\mu I)={\rm rank}(A+\mu I)
\end{eqnarray*}
and $\frac{\beta}{\alpha-\beta}x^\top{\rm j}=\frac{\beta}{\alpha-\beta}{\rm j}^\top(A+\mu I)^\#{\rm j}>-1$, i.e., ${\rm j}^\top(A+\mu I)^\#{\rm j}<\frac{\alpha-\beta}{-\beta}$. If part (4) in Lemma \ref{lem2.4} holds, then
\begin{eqnarray*}
{\rm rank}(A+\mu I)={\rm in}_+(A+\mu I)={\rm in}_+(U^\top U)+1=r+1
\end{eqnarray*}
and $\frac{\beta}{\alpha-\beta}x^\top{\rm j}=\frac{\beta}{\alpha-\beta}{\rm j}^\top(A+\mu I)^\#{\rm j}=-1$, i.e., ${\rm j}^\top(A+\mu I)^\#{\rm j}=\frac{\alpha-\beta}{-\beta}$.

Conversely, if an $n$-vertex graph $G$ satisfies parts (1) and (2) in this theorem, then from the above arguments, $A_G+\mu I+\frac{\beta}{\alpha-\beta}J$ is positive semidefinite with rank $r$. Hence there exists a spherical $\{\alpha,\beta\}$-code $S$ with rank $r$ whose Gram matrix is $(\alpha-\beta)\left(A_G+\mu I+\frac{\beta}{\alpha-\beta}J\right)$, i.e., $\Gamma_\alpha(S)=G$.
\end{proof}

\begin{rem}
An eigenvalue $\lambda$ of a graph $G$ is called a main eigenvalue of $G$ if $\lambda$ has an eigenvector not orthogonal to ${\rm j}$. By Theorem \ref{thm3.1}, we know that the smallest eigenvalue of $\Gamma_\alpha(S)$ is at least $-\mu$. If the smallest eigenvalue of $\Gamma_\alpha(S)$ equals to $-\mu$, then ${\rm j}\in R(A+\mu I)$ is equivalent to that every eigenvector of $-\mu$ is orthogonal to ${\rm j}$, i.e., $-\mu$ is a non-main eigenvalue of $\Gamma_\alpha(S)$. Since ${\rm j}\in R(A+\mu I)$, we have ${\rm j}^\top(A+\mu I)^\#{\rm j}={\rm j}^\top N{\rm j}$ for any matrix $N$ satisfying $(A+\mu I)N(A+\mu I)=A+\mu I$, i.e., $N$ is any $\{1\}$-inverse of $A+\mu I$.
\end{rem}


For the $\alpha$-graph of a spherical $\{\alpha,\beta\}$-code, its induced graphs have the following properties.
\begin{thm}
Suppose that $-1\leq\beta<\alpha<1$, $\beta<0$, and $\mu=\frac{1-\beta}{\alpha-\beta}$. Let $G$ be the $\alpha$-graph of a spherical $\{\alpha,\beta\}$-code. For any $t$-vertex induced graph $H$ of $G$, we have
\begin{eqnarray*}
t^2\leq(2|E(H)|+t\mu){\rm j}^\top(A_G+\mu I)^\#{\rm j}\leq\frac{\alpha-\beta}{-\beta}(2|E(H)|+t\mu).
\end{eqnarray*}
\end{thm}
\begin{proof}
By Theorem \ref{thm3.1}, we know that 
\begin{eqnarray*}
{\rm j}\in R(A_G+\mu I),~{\rm j}^\top(A_G+\mu I)^\#{\rm j}\leq\frac{\alpha-\beta}{-\beta}.
\end{eqnarray*}
So we have 
\begin{eqnarray*}
(A_G+\mu I)^{\frac{1}{2}}(A_G+\mu I)^{-\frac{1}{2}}{\rm j}={\rm j}.
\end{eqnarray*}
Let $x$ be the vector such that $(x)_i=1$ if $i\in V(H)$, and $(x)_i=0$ otherwise. By the Cauchy-Schwarz inequality, we have
\begin{eqnarray*}
t^2&=&(x^\top{\rm j})^2=(x^\top(A_G+\mu I)^{\frac{1}{2}}(A_G+\mu I)^{-\frac{1}{2}}{\rm j})^2\leq x^\top(A_G+\mu I)x{\rm j}^\top(A_G+\mu I)^\#{\rm j}\\
&\leq&\frac{\alpha-\beta}{-\beta}(2|E(H)|+t\mu).
\end{eqnarray*}
\end{proof}
The following result is derived from the above theorem.
\begin{cor}\label{cor3.2}
Suppose that $-1\leq\beta<\alpha<1$, $\beta<0$, and $\mu=\frac{1-\beta}{\alpha-\beta}$. Let $G$ be the associated $\alpha$-graph of a spherical $\{\alpha,\beta\}$-code. Then the independence number $t$ of $G$ satisfies
\begin{eqnarray*}
t\leq\mu{\rm j}^\top(A_G+\mu I)^\#{\rm j}\leq\frac{1-\beta}{-\beta}.
\end{eqnarray*}
\end{cor}

Let $K_t$ denote the complete graph with $t$ vertices. A graph $G$ is called $K_{r+1}$-free if $G$ does not contain $K_{r+1}$ as a subgraph.
\begin{thm}
Suppose that $-1\leq\beta<\alpha<1$, $\beta<0$. Let $G$ be the $\alpha$-graph of a spherical $\{\alpha,\beta\}$-code with rank $r$. If $G\neq K_{r+1}$, then $G$ is $K_{r+1}$-free.
\end{thm}
\begin{proof}
Suppose that $K_{r+1}$ is a subgraph of $G$. Since $G\neq K_{r+1}$, we can write $A_G+\mu I$ as 
\begin{eqnarray*}
A_G+\mu I=\begin{pmatrix}J_{r+1}+(\mu-1)I&B\\B^\top&C\end{pmatrix},
\end{eqnarray*}
where $\mu=\frac{1-\beta}{\alpha-\beta}>1$. By Theorem \ref{thm3.1}, we have ${\rm j}\in R(A_G+\mu I)$ and ${\rm rank}(A_G+\mu I)\leq r+1$. Since $J_{r+1}+(\mu-1)I$ is nonsingular with ${\rm rank}(J_{r+1}+(\mu-1)I)=r+1$, we have 
\begin{eqnarray*}
{\rm rank}(A_G+\mu I)&=&{\rm rank}(J_{r+1}+(\mu-1)I),\\
C&=&B^\top(J_{r+1}+(\mu-1)I)^{-1}B. 
\end{eqnarray*}
Then $N=\begin{pmatrix}(J_{r+1}+(\mu-1)I)^{-1}&0\\0&0\end{pmatrix}$ is a $\{1\}$-inverse of $A_G+\mu I$. By ${\rm j}\in R(A_G+\mu I)$ we get $(A_G+\mu I)N{\rm j}={\rm j}$. By computation, we have
\begin{eqnarray*}
(A_G+\mu I)N{\rm j}=\begin{pmatrix}{\rm j}\\B^\top(J_{r+1}+(\mu-1)I)^{-1}{\rm j}\end{pmatrix}=\begin{pmatrix}{\rm j}\\(r+\mu)^{-1}B^\top{\rm j}\end{pmatrix}.
\end{eqnarray*}
We can not choose a $(0,1)$-matrix $B$ such that $(r+\mu)^{-1}B^\top{\rm j}={\rm j}$, a contradiction. Hence $G$ is $K_{r+1}$-free.
\end{proof}

For a vertex $u$ of $G$, let $N_G(u)=\{v:\{v,u\}\in E(G)\}$ denote the open neighborhood of $u$, and let $N_G[u]=N_G(u)\cup\{u\}$ denote the closed neighborhood of $u$. Let $G_u$ denote the subgraph of $G$ induced by $N_G(u)$, and let $G-N_G[u]$ denote the induced subgraph of $G$ by deleting all vertices in $N_G[u]$.
\begin{thm}
Suppose that $-1\leq\beta<\alpha<1$, $\beta<0$. Let $G$ be the $\alpha$-graph of a spherical $\{\alpha,\beta\}$-code. For any vertex $u$ of $G$, the induced subgraphs $G_u$ and $G-N_G[u]$ satisfy
\begin{eqnarray*}
{\rm j}^\top(A_{G_u}+\mu I)^\#{\rm j}&\leq&\frac{\alpha-\beta}{\alpha^2-\beta},\\
{\rm rank}(A_{G_u}+\mu I)&\leq&{\rm rank}(A_G+\mu I)-1,\\
{\rm j}^\top(A_{G-N_G[u]}+\mu I)^\#{\rm j}&\leq&\frac{\alpha-\beta}{-\beta(1-\beta)},\\
{\rm rank}(A_{G-N_G[u]}+\mu I)&\leq&{\rm rank}(A_G+\mu I)-1.
\end{eqnarray*}
\end{thm}
\begin{proof}
Let $H$ be the subgraph induced by $N_G[u]$. We can write $A_H+\mu I$ as 
\begin{eqnarray*}
A_H+\mu I=\begin{pmatrix}A_{G_u}+\mu I&{\rm j}\\{\rm j}^\top&\mu\end{pmatrix}.
\end{eqnarray*}
Since ${\rm j}\in R(A_{G_u}+\mu I)$, we have
\begin{eqnarray*}
{\rm rank}(A_H+\mu I)={\rm rank}(A_{G_u}+\mu I)+{\rm rank}(\mu-{\rm j}^\top(A_{G_u}+\mu I)^\#{\rm j})\leq{\rm rank}(A_G+\mu I).
\end{eqnarray*}
Let $x={\rm j}^\top(A_{G_u}+\mu I)^\#{\rm j}$. Since $A_H+\mu I$ is positive semidefinite, we have $\mu-x\geq0$. If $x=\mu$, then $N=\begin{pmatrix}(A_{G_u}+\mu I)^\#&0\\0&0\end{pmatrix}$ is a $\{1\}$-inverse of $A_H+\mu I$ and $(A_H+\mu I)N{\rm j}={\rm j}$. By computation, we have $(A_H+\mu I)N{\rm j}=\begin{pmatrix}{\rm j}\\ \mu\end{pmatrix}\neq{\rm j}$, a contradiction. So $x<\mu$ and
\begin{eqnarray*}
{\rm rank}(A_{G_u}+\mu I)\leq{\rm rank}(A_G+\mu I)-1.
\end{eqnarray*}
By $x<\mu$ and ${\rm j}\in R(A_{G_u}+\mu I)$, we can get
\begin{eqnarray*}
(A_H+\mu I)^{(1)}=\begin{pmatrix}(A_{G_u}+\mu I)^\#+(\mu-x)^{-1}(A_{G_u}+\mu I)^\#{\rm j}{\rm j}^\top(A_{G_u}+\mu I)^\#&-(\mu-x)^{-1}(A_{G_u}+\mu I)^\#{\rm j}\\-(\mu-x)^{-1}(A_{G_u}+\mu I)^\#&(\mu-x)^{-1}\end{pmatrix}
\end{eqnarray*}
Then 
\begin{eqnarray*}
{\rm j}^\top(A_H+\mu I)^\#{\rm j}={\rm j}^\top(A_H+\mu I)^{(1)}{\rm j}=x+(\mu-x)^{-1}(x-1)^2\leq\frac{\alpha-\beta}{-\beta}.
\end{eqnarray*}
Take $p=\frac{\alpha-\beta}{-\beta}$, then $x\leq\frac{p\mu-1}{\mu+p-2}=\frac{\alpha-\beta}{\alpha^2-\beta}$.

Let $H_1$ be the subgraph of $G$ induced by $\{u\}\cup\{v:v\notin N_G[u]\}$. Then $A_{H_1}+\mu I=\begin{pmatrix}A_{H_2}+\mu I&0\\0&\mu\end{pmatrix}$, where $H_2=G-N_G[u]$. Hence
\begin{eqnarray*}
{\rm rank}(A_{H_1}+\mu I)&=&{\rm rank}(A_{H_2}+\mu I)+1\leq{\rm rank}(A_G+\mu I),\\
{\rm j}^\top(A_{H_1}+\mu I)^\#{\rm j}&=&{\rm j}^\top(A_{H_2}+\mu I)^\#{\rm j}+\mu^{-1}\leq\frac{\alpha-\beta}{-\beta},\\
{\rm j}^\top(A_{H_2}+\mu I)^\#{\rm j}&\leq&\frac{\alpha-\beta}{-\beta}-\mu^{-1}=\frac{\alpha-\beta}{-\beta(1-\beta)}.
\end{eqnarray*}
\end{proof}


\section{The associated $\beta$-graphs of spherical two-distance sets}
Let $S=\{u_1,\ldots,u_n\}$ be a spherical $\{\alpha,\beta\}$-code ($-1\leq\beta<\alpha<1$), and let $\Gamma_\beta(S)$ denote the associated $\beta$-graph with vertex set $S$ and edge set $E(\Gamma_\beta(S))=\{u_iu_j:u_i^\top u_j=\beta,i\neq j\}$. There exists correspondence between spherical $\{0,\beta\}$-codes and $\beta$-graphs as follows.
\begin{pro}\label{pro4.1}
Suppose that $-1\leq\beta<0$. If $S$ is a spherical $\{0,\beta\}$-code with size $n$ and rank $r$, then one of the following holds:\\
(1) $\lambda_1(\Gamma_\beta(S))<\frac{1}{-\beta}$ and $n=r$.\\
(2) $\lambda_1(\Gamma_\beta(S))=\frac{1}{-\beta}$ is an eigenvalue of $\Gamma_\beta(S)$ with codimension $r$.\\
Conversely, if an $n$-vertex graph $G$ satisfies (1) or (2), then there exists a spherical $\{0,\beta\}$-code $S$ with rank $r$ such that $\Gamma_\beta(S)=G$.
\end{pro}
\begin{proof}
Suppose that $S=\{u_1,\ldots,u_n\}$ is a spherical $\{0,\beta\}$-code with rank $r$, and let $\lambda=\frac{1}{-\beta}$. Let $U$ be the matrix whose $i$-th column is $u_i$, then the Gram matrix of $S$ is $U^\top U=-\beta(\lambda I-A)$, where $A$ is the adjacency matrix of the $\beta$-graph $G=\Gamma_\beta(S)$. Then $\lambda I-A$ is positive semidefinite with ${\rm rank}(\lambda I-A)={\rm rank}(U^\top U)=r$ and $\lambda_1(G)\leq\lambda$.

If $\lambda_1(G)<\lambda$, then $n={\rm rank}(\lambda I-A)=r$, so $S$ satisfies part (1). If $\lambda_1(G)=\lambda$, then $\lambda$ is an eigenvalue of $G$ with codimension $r$, so $S$ satisfies part (2).

Conversely, if an $n$-vertex graph $G$ satisfies (1) or (2), then from the above arguments, there exists a spherical $\{\alpha,\beta\}$-code $S$ with rank $r$ such that $\Gamma_\beta(S)=G$.
\end{proof}
For $-1\leq\beta<\alpha<1$, $\alpha>0$, we give the correspondence between spherical $\{\alpha,\beta\}$-codes and $\beta$-graphs as follows.
\begin{thm}\label{thm4.2}
Suppose that $-1\leq\beta<\alpha<1$, $\alpha>0$, and let $\lambda=\frac{1-\alpha}{\alpha-\beta}$. If $S$ is a spherical $\{\alpha,\beta\}$-code with size $n$ and rank $r$, then one of the following holds:\\
(1) $\lambda_1(\Gamma_\beta(S))<\lambda$ and $n=r$.\\
(2) $\lambda_1(\Gamma_\beta(S))=\lambda$ is an eigenvalue of $\Gamma_\beta(S)$ with codimension $r-1$.\\
(3) Let $A$ be the adjacency matrix of $\Gamma_\beta(S)$, then
\begin{eqnarray*}
{\rm in}_-(\lambda I-A)=1,~{\rm j}\in R(\lambda I-A),~{\rm j}^\top(\lambda I-A)^\#{\rm j}\leq\frac{\alpha-\beta}{-\alpha},\\
{\rm rank}(\lambda I-A)=\begin{cases}r+1~~~~\mbox{if}~~{\rm j}^\top(\lambda I-A)^\#{\rm j}=\frac{\alpha-\beta}{-\alpha},\\
r~~~~~~~~~\mbox{otherwise}.\end{cases}
\end{eqnarray*}
Conversely, if an $n$-vertex graph $G$ satisfies one of (1), (2) and (3), then there exists a spherical $\{\alpha,\beta\}$-code $S$ with rank $r$ such that $\Gamma_\beta(S)=G$.
\end{thm}
\begin{proof}
Suppose that $S=\{u_1,\ldots,u_n\}$ is a spherical $\{\alpha,\beta\}$-code with rank $r$. Let $U$ be the matrix whose $i$-th column is $u_i$, then the Gram matrix of $S$ is
\begin{eqnarray*}
U^\top U=(\alpha-\beta)(\lambda I-A)+\alpha J=(\alpha-\beta)\left(\lambda I-A+\frac{\alpha}{\alpha-\beta}{\rm j}{\rm j}^\top\right),
\end{eqnarray*}
where $A$ is the adjacency matrix of $G=\Gamma_\beta(S)$. Then ${\rm in}_+(U^\top U)=r$ and ${\rm in}_-(U^\top U)=0$. By Lemma \ref{lem2.5}, we know that ${\rm in}_-(\lambda I-A)\leq1$. We only need to consider the cases: (i) $\lambda_1(G)<\lambda$; (ii) $\lambda_1(G)=\lambda$; (iii) $\lambda_1(G)>\lambda$ and $\lambda_2(G)\leq\lambda$.

If $\lambda_1(G)<\lambda$, then $\lambda I-A$ and $U^\top U=(\alpha-\beta)(\lambda I-A)+\alpha J$ are both positive definite. Then $n={\rm rank}(U^\top U)=r$. So $S$ satisfies part (1) in this theorem.

If $\lambda_1(G)=\lambda$, then ${\rm in}_-(\lambda I-A)={\rm in}_-(U^\top U)=0$. So part (1) or (2) in Lemma \ref{lem2.5} holds. Notices that ${\rm j}\notin R(\lambda I-A)$, because $\lambda_1(G)=\lambda$ has a nonnegative eigenvector which is not orthogonal to ${\rm j}$. Hence part (1) in Lemma \ref{lem2.5} holds and ${\rm rank}(\lambda I-A)={\rm in}_+(\lambda I-A)={\rm in}_+(U^\top U)-1=r-1$. So $S$ satisfies part (2) in this theorem.

If $\lambda_1(G)>\lambda$ and $\lambda_2(G)\leq\lambda$, then ${\rm in}_-(\lambda I-A)=1$. Since ${\rm in}_-(U^\top U)=0$, part (3) or (4) in Lemma \ref{lem2.5} holds, and in both cases we have ${\rm j}\in R(A+\mu I)$. Then ${\rm j}=(\lambda I-A)x$ for some $x\in\mathbb{R}^n$ and $(\lambda I-A)(\lambda I-A)^\#{\rm j}={\rm j}$. So we have
\begin{eqnarray*}
x^\top{\rm j}=x^\top(\lambda I-A)(\lambda I-A)^\#{\rm j}={\rm j}^\top(\lambda I-A)^\#{\rm j}.
\end{eqnarray*}
If part (3) in Lemma \ref{lem2.5} holds, then
\begin{eqnarray*}
{\rm rank}(\lambda I-A)={\rm rank}(U^\top U)=r
\end{eqnarray*}
and $\frac{\alpha}{\alpha-\beta}x^\top{\rm j}=\frac{\alpha}{\alpha-\beta}{\rm j}^\top(A+\mu I)^\#{\rm j}<-1$, i.e., ${\rm j}^\top(A+\mu I)^\#{\rm j}<\frac{\alpha-\beta}{-\alpha}$. If part (4) in Lemma \ref{lem2.5} holds, then
\begin{eqnarray*}
r={\rm in}_+(U^\top U)={\rm in}_+(\lambda I-A),~{\rm in}_-(\lambda I-A)=1,~{\rm rank}(\lambda I-A)=r+1
\end{eqnarray*}
and $\frac{\alpha}{\alpha-\beta}x^\top{\rm j}=\frac{\alpha}{\alpha-\beta}{\rm j}^\top(\lambda I-A)^\#{\rm j}=-1$, i.e., ${\rm j}^\top(\lambda I-A)^\#{\rm j}=\frac{\alpha-\beta}{-\alpha}$. So $S$ satisfies part (3) in this theorem.

Conversely, if an $n$-vertex graph $G$ satisfies one of (1), (2) and (3) in this theorem, then from the above arguments, $\lambda I-A_G+\frac{\alpha}{\alpha-\beta}J$ is positive semidefinite with rank $r$. Hence there exists a spherical $\{\alpha,\beta\}$-code $S$ with rank $r$ whose Gram matrix is $(\alpha-\beta)\left(\lambda I-A_G+\frac{\alpha}{\alpha-\beta}J\right)$, i.e., $\Gamma_\beta(S)=G$.
\end{proof}

\begin{rem}
In Theorem \ref{thm4.2}(3), ${\rm in}_-(\lambda I-A)=1$ means that $\lambda_1(\Gamma_\beta(S))>\lambda,\lambda_2(\Gamma_\beta(S))\leq\lambda$. If $\lambda_2(\Gamma_\beta(S))=\lambda$, then ${\rm j}\in R(\lambda I-A)$ is equivalent to that every eigenvector of $\lambda$ is orthogonal to ${\rm j}$. Since ${\rm j}\in R(\lambda I-A)$, we have ${\rm j}^\top(\lambda I-A)^\#{\rm j}={\rm j}^\top N{\rm j}$ for any matrix $N$ satisfying $(\lambda I-A)N(\lambda I-A)=\lambda I-A$, i.e., $N$ is any $\{1\}$-inverse of $\lambda I-A$.
\end{rem}

\section{Bounds on spherical two-distance sets}
Let $\mathcal{G}_{\mu,r}$ be the set of connected graph $G$ such that $\lambda(G)\geq-\mu$, ${\rm j}\in R(A_G+\mu I)$ and ${\rm rank}(A_G+\mu I)\leq r$.
\begin{thm}
For a constant $\mu>1$, we have
\begin{eqnarray*}
\max_{\substack{G\in\mathcal{G}_{\mu,d+1},\\G\neq K_{d+1}}}\left\{\left\lfloor\frac{d+1}{{\rm rank}(A_G+\mu I)}\right\rfloor|V(G)|+q\right\}\leq\max_{\substack{\frac{1-\beta}{\alpha-\beta}=\mu,\\-1\leq\beta<\alpha<1,\\ \beta<0}}N_{\alpha,\beta}(d)\leq(d+1)\max_{G\in\mathcal{G}_{\mu,d+1}}\frac{|V(G)|}{{\rm rank}(A_G+\mu I)},
\end{eqnarray*}
where $q=d+1-{\rm rank}(A_G+\mu I)\left\lfloor\frac{d+1}{{\rm rank}(A_G+\mu I)}\right\rfloor$.
\end{thm}
\begin{proof}
For any graph $G\in\mathcal{G}_{\mu,d+1}\setminus\{K_{d+1}\}$, let $H=kG\cup K_q$, where $k=\left\lfloor\frac{d+1}{{\rm rank}(A_G+\mu I)}\right\rfloor$, $q=d+1-{\rm rank}(A_G+\mu I)\left\lfloor\frac{d+1}{{\rm rank}(A_G+\mu I)}\right\rfloor$. Then $H$ has at least a pair of non-adjacent vertices. By Corollary 3.3, we have
\begin{eqnarray*}
\mu{\rm j}^\top(A_H+\mu I)^\#{\rm j}\geq2.
\end{eqnarray*}
Then there exist $-1\leq\beta_0<\alpha_0<1$ ($\beta_0<0$) such that $\frac{1-\beta_0}{\alpha_0-\beta_0}=\mu$ and ${\rm j}^\top(A_H+\mu I)^\#{\rm j}=\frac{\alpha_0-\beta_0}{-\beta_0}$. By Theorem \ref{thm3.1}, there exists a spherical $\{\alpha_0,\beta_0\}$-code $S$ in $\mathbb{R}^d$ such that $\Gamma_\alpha(S)=H$. Hence
\begin{eqnarray*}
\max_{\frac{1-\beta}{\alpha-\beta}=\mu}N_{\alpha,\beta}(d)\geq|V(H)|=k|V(G)|+q
\end{eqnarray*}
for any $G\in\mathcal{G}_{\mu,d+1}\setminus\{K_{d+1}\}$.

Let $S$ be any spherical $\{\alpha,\beta\}$-code in $\mathbb{R}^d$ such that $\frac{1-\beta}{\alpha-\beta}=\mu$. Suppose that the $\alpha$-graph $G$ of $S$ has $t$ connected components $G_1,\ldots,G_t$. Then
\begin{eqnarray*}
\frac{|V(G)|}{d+1}&\leq&\frac{|V(G)|}{{\rm rank}(A_G+\mu I)}=\frac{\sum_{i=1}^t|V(G_i)|}{\sum_{i=1}^t{\rm rank}(A_{G_i}+\mu I)}\leq\max_{1\leq i\leq t}\frac{|V(G_i)|}{{\rm rank}(A_{G_i}+\mu I)}\\
&\leq&\max_{H\in\mathcal{G}_{\mu,d+1}}\frac{|V(H)|}{{\rm rank}(A_H+\mu I)}.
\end{eqnarray*}
\end{proof}

\subsection{Bounds derived from local structures}

For $p>0$, $\mu>1$ and positive integer $r$, let $N(r,p,\mu)$ (resp. $N^*(r,p,\mu)$) denote the maximum number of vertices in a graph $G$ satisfying $\lambda(G)\geq-\mu$, ${\rm j}\in R(A+\mu I)$, ${\rm rank}(A+\mu I)\leq r$ and ${\rm j}^\top(A+\mu I)^\#{\rm j}<p$ (resp. ${\rm j}^\top(A+\mu I)^\#{\rm j}=p$). Clearly, $N(r,p,\mu)$ is an increasing function with respect to $r$ and $p$, and $N^*(r,p,\mu)$ is an increasing function with respect to $r$.

From Theorem \ref{thm3.1}, we have $N_{\alpha,\beta}(d)=\max\{N(d,\frac{\alpha-\beta}{-\beta},\frac{1-\beta}{\alpha-\beta}),N^*(d+1,\frac{\alpha-\beta}{-\beta},\frac{1-\beta}{\alpha-\beta})\}$. \begin{thm}
Suppose that $-1\leq\beta<\alpha<1$, $\beta<0$. Then
\begin{eqnarray*}
N_{\alpha,\beta}(d)&\leq&\frac{\alpha-\beta}{-\beta}\left(f(\alpha,\beta,d)+\frac{1-\beta}{\alpha-\beta}\right),\\
f(\alpha,\beta,d)&\leq&N_{\alpha_0,\beta_0}(d),
\end{eqnarray*}
where
\begin{eqnarray*}
f(\alpha,\beta,d)&=&\max\left\{N\left(d,\frac{\alpha-\beta}{\alpha^2-\beta},\frac{1-\beta}{\alpha-\beta}\right),N^*\left(d,\frac{\alpha-\beta}{\alpha^2-\beta},\frac{1-\beta}{\alpha-\beta}\right)\right\},\\
\alpha_0&=&\frac{\alpha}{1+\alpha},\beta_0=\frac{\beta-\alpha^2}{1-\alpha^2}.
\end{eqnarray*}
\end{thm}
\begin{proof}
Let $G$ be the $\alpha$-graph of a spherical $\{\alpha,\beta\}$-code in $\mathbb{R}^d$ such that $|V(G)|=N_{\alpha,\beta}(d)$, then 
\begin{eqnarray*}
{\rm rank}(A_G+\frac{1-\beta}{\alpha-\beta}I)\leq d+1. 
\end{eqnarray*}
Let $u$ be a vertex in $G$ with maximum degree. By Theorem 3.2, we get
\begin{eqnarray*}
N_{\alpha,\beta}(d)\leq\frac{\alpha-\beta}{-\beta}\left(|V(G_u)|+\frac{1-\beta}{\alpha-\beta}\right).
\end{eqnarray*}
Theorem 3.5 implies that $|V(G_u)|\leq f(\alpha,\beta,d)$.

Take $\alpha_0=\frac{\alpha}{1+\alpha},\beta_0=\frac{\beta-\alpha^2}{1-\alpha^2}$, then $\frac{\alpha_0-\beta_0}{-\beta_0}=\frac{\alpha-\beta}{\alpha^2-\beta}$ and $\frac{1-\beta_0}{\alpha_0-\beta_0}=\frac{1-\beta}{\alpha-\beta}$. So $f(\alpha,\beta,d)\leq N_{\alpha_0,\beta_0}(d)$.
\end{proof}



\subsection{Spherical two-distance sets with small $\frac{\alpha-\beta}{-\beta}$}

\vspace{3mm}

When $\frac{\alpha-\beta}{-\beta}<1+\frac{1}{d-1}$, we have the following bound on $N_{\alpha,\beta}(d)$.
\begin{thm}
Suppose that $-1\leq\beta<\alpha<1$, $\beta<0$. If $\frac{\alpha-\beta}{-\beta}<1+\frac{1}{d-1}$, then
\begin{eqnarray*}
N_{\alpha,\beta}(d)\leq\max\left\{d+1,\left(\frac{-\alpha}{\alpha-\beta}+d^{-1}\right)^{-1}\frac{1-\beta}{\alpha-\beta}\right\}.
\end{eqnarray*}
\end{thm}
\begin{proof}
Let $G$ be an $n$-vertex $\alpha$-graph of a spherical $\{\alpha,\beta\}$-code in $\mathbb{R}^d$ such that $n=N_{\alpha,\beta}(d)$. By, we obtain
\begin{eqnarray*}
n\leq\frac{\alpha-\beta}{-\beta}\left(\frac{2|E(G)|}{n}+\mu\right).
\end{eqnarray*}
If $G\neq K_{d+1}$, then by Theorem 3.4, we know that $G$ is $K_{d+1}$-free. By the Tur\'{a}n theorem, we have
\begin{eqnarray*}
n\leq\frac{\alpha-\beta}{-\beta}((1-d^{-1})n+\mu),\\
n\leq\left(\frac{-\alpha}{\alpha-\beta}+d^{-1}\right)^{-1}\mu.
\end{eqnarray*}
\end{proof}

\subsection{Spherical two-distance sets with large $\frac{1-\beta}{\alpha-\beta}$}

\vspace{3mm}

When $\frac{1-\beta}{\alpha-\beta}$ is large with respect to $d$, we have the following bound on $N_{\alpha,\beta}(d)$.
\begin{thm}
Suppose that $-1\leq\beta<\alpha<1$, $\beta<0$. If $\frac{1-\beta}{\alpha-\beta}>\sqrt{\lfloor\frac{d+2-k}{2}\rfloor\lceil\frac{d+2-k}{2}\rceil}$ for an integer $k$ ($0\leq k\leq d+1$), then
\begin{eqnarray*}
N_{\alpha,\beta}(d)\leq2^k(d+2-k)-1.
\end{eqnarray*}
\end{thm}
\begin{proof}
Let $G$ be the $\alpha$-graph of a spherical $\{\alpha,\beta\}$-code in $\mathbb{R}^d$ such that $|V(G)|=N_{\alpha,\beta}(d)$, then
\begin{eqnarray*}
{\rm rank}(A_G+\frac{1-\beta}{\alpha-\beta}I)\leq d+1.
\end{eqnarray*}
Let $G_{0,0}=G$, and let $G_{i,j}$ ($0\leq i\leq k,0\leq j\leq2^i-1$) be a family of graphs such that $G_{i+1,2j+1}$ is the subgraph of $G_{i,j}$ induced by $N_{G_{i,j}}(u_{ij})$ and $G_{i+1,2j+2}$ is the subgraph of $G_{i,j}$ induced by $V(G_{i,j})\setminus N_{G_{i,j}}[u_{ij}]$, where $u_{ij}$ is a vertex of $G_{i,j}$. Then $|V(G_{i,j})|=1+|V(G_{i+1,2j+1})|+|V(G_{i+1,2j+2})|$ and
\begin{eqnarray*}
|V(G)|=\sum_{j=0}^{2^k-1}|V(G_{k,j})|+\sum_{s=0}^{k-1}2^s=2^k-1+\sum_{j=0}^{2^k-1}|V(G_{k,j})|.
\end{eqnarray*}
By Theorem 3.5, we have
\begin{eqnarray*}
{\rm rank}(A_{G_{k,j}}+\frac{1-\beta}{\alpha-\beta}I)\leq{\rm rank}(A_G+\frac{1-\beta}{\alpha-\beta}I)-k\leq d+1-k
\end{eqnarray*}
for $0\leq j\leq2^k-1$. 

If $|V(G_{k,j})|>d+1-k$, then $-\frac{1-\beta}{\alpha-\beta}$ is the smallest eigenvalue of $G_{k,j}$. By Lemma \ref{lem2.1}, we get $\frac{1-\beta}{\alpha-\beta}\leq\sqrt{\lfloor\frac{d+2-k}{2}\rfloor\lceil\frac{d+2-k}{2}\rceil}$, a contradiction. So $|V(G_{k,j})|\leq d+1-k$ for $0\leq j\leq2^k-1$. Hence
\begin{eqnarray*}
|V(G)|=2^k-1+\sum_{j=0}^{2^k-1}|V(G_{k,j})|\leq2^k(d+1-k)+2^k-1.
\end{eqnarray*}
\end{proof}





\end{CJK*}
\end{spacing}
\end{document}